\documentclass[12pt,oneside]{amsart}
\usepackage[normalem]{ulem} % for \sout strikeout text
\usepackage{latexsym,amssymb,amsmath,graphicx,hyperref}
\usepackage{subfig}
\usepackage{caption}
\usepackage{verbatim}
\numberwithin{equation}{section}
\usepackage{tikz}
\usetikzlibrary{shapes}

\usepackage{fp}
\usepackage{rotating}
\usetikzlibrary{arrows}
\usepackage{longtable}
\usepackage{mathtools}
\usepackage{centernot}

\DeclarePairedDelimiter\floor{\lfloor}{\rfloor}

\newtheorem{theorem}{Theorem}[section]
\newtheorem{lemma}[theorem]{Lemma}

\theoremstyle{definition}
\newtheorem{definition}[theorem]{Definition}
\newtheorem{remark}[theorem]{Remark}

\newcommand{\PIVAL}{3.14159265358979323846264338} %There is no easily accessible constant
\newcounter{i} % Counter for the while loop
\newcommand{\Circulant}[2] { %Accepts 2 argument, the number of nodes, and the list of vertices
	\begin{tikzpicture}
	\setcounter{i}{0}
	\whiledo{\value{i}<#1}{ %Start counting through the nodes
		\FPmul\tempA{2}{\thei} %For the X,Y Formulae start with 2*i
		\FPdiv\tempB{\PIVAL}{#1} % pi/(nodes)
		\FPmul\tempC{\tempA}{\tempB} % the product of the previous
		\FPcos\varX{\tempC} % cos that for x
		\FPsin\varY{\tempC} % sin it for y
		\stepcounter{i} % Count up 1
		\FPround\varX{\varX}{3}
		\FPround\varY{\varY}{3}
		\node (\thei) at (\varX,\varY)[place]{ }; % Draw the nodes
		\foreach \x in {#2} { % For all the variables in the Vertice list
			\pgfmathparse{mod(\x+\thei,#1)} % Do the modular count forward
			\let\tempB\pgfmathresult
			\pgfmathparse{mod(\thei-\x,#1)} % And the modular count backward
			\let\tempA\pgfmathresult
			\ifthenelse{\lengthtest{\tempA pt < 1 pt}}{\FPadd\tempA{\tempA}{#1}}{}
			\ifthenelse{\lengthtest{\tempB pt < 1 pt}}{\FPadd\tempB{\tempB}{#1}}{}
			\ifthenelse{\lengthtest{\tempA pt > \thei pt}}{}{\ifthenelse{\thei = \tempA}{}{\draw [] (\thei) to (\tempA)}};
			\ifthenelse{\lengthtest{\tempB pt > \thei pt}}{}{\ifthenelse{\thei = \tempB}{}{\draw [] (\thei) to (\tempB)}};
		}
	}
	\end{tikzpicture}
}

\begin{document}

\tikzstyle{place}=[draw,circle,minimum size=0.5mm,inner sep=1pt,outer sep=-1.1pt]

%%%%%%%%%%%%%%%%%%%%%%%%%%%%%%%%%%%%%%%%%%%%%%%%%%%%%%%%%%%%%%%%%%%%%%%%%%%%%%

\title[Shellability, vertex decomposability, and the lexicographical
product]{Shellability, vertex decomposability, and lexicographical
products of graphs}
\thanks{Last updated: \today}
\thanks{Research of the two authors supported in part by 
NSERC Discovery Grants.}

\author{Kevin N. Vander Meulen}
\address{Department of Mathematics\\
Redeemer University College, Ancaster, ON, L9K 1J4, Canada}
\email{kvanderm@redeemer.ca}

\author{Adam Van Tuyl}
\address{Department of Mathematics \& Statistics\\
McMaster University \\
Hamilton, ON, L8S 4L8, Canada}
\email{vantuyl@math.mcmaster.ca}


\begin{abstract}
We investigate when the independence
complex of $G[H]$, the lexicographical product of two graphs $G$ and $H$,
is either vertex decomposable or shellable. 
As an application,  we construct an infinite family of 
graphs with the property that every graph in this family has the 
property that the independence complex of each graph is shellable, but not 
vertex decomposable.  
\end{abstract}

\maketitle

%%%%%%%%%%%%%%%%%%%%%%%%%%%%%%%%%%%%%%%%%%%%%%%%%%%%%%%%%%%%%%%%%%%%%%%%%%%%%%%
\section{Introduction}\label{sec:intro}

Let $G = (V_G,E_G)$ and $H = (V_H,E_H)$ be two finite simple graphs.
There are a number of constructions in the
literature that enable
one to make a ``product'' of two graphs, that is, a new graph
on the vertex set $V_G \times V_H$. 
In this paper we are interested in the lexicographical product.
The {\it lexicographical product} of $G$
and $H$, denoted $G[H]$, is the graph with the vertex set $V_G \times V_H$, 
such that 
$(w,x)$ and $(y,z)$ are adjacent if $\{w,y\} \in E_G$ or if $w=y$ and
$\{x,z\} \in E_H$.  

Given some property that both $G$ and $H$ possess, 
it is then natural to ask if $G[H]$ also possess this property.  
The property of being well-covered is an example 
of such an inherited property.  
Recall that a subset $W \subseteq V_G$ of a graph $G$
is a {\it vertex cover} if $e \cap W \neq \emptyset$ for all $e \in E_G$.  
A graph is {\it well-covered} if every minimal (ordered with respect to
inclusion) vertex cover has the same cardinality.  
Topp and Volkmann \cite{TV} showed that 
$G$ and $H$ are well-covered if and only if $G[H]$ is well-covered.

In this note we focus on the independence complex of $G[H]$.
Recall that a subset $W \subseteq V_G$ is an {\it independent
set} if for all $e \in E_G$, $e \not\subseteq W$.  Equivalently,
$W \subseteq V_G$ is an independent set if and only if $V_G \setminus W$
is a vertex cover of $G$.  The independence complex of a graph
$G$, denoted ${\rm Ind}(G)$, is the simplicial complex
\[{\rm Ind}(G) = \{W \subseteq V_G ~|~ \mbox{$W$ is a independent set}\}.\]
Because of the duality between vertex covers and independent sets,
${\rm Ind}(G)$ is pure (see the next section) 
if and only if $G$ is well-covered.  
Topp and Volkmann's result can be restated
as saying ${\rm Ind}(G[H])$
is pure if and only if ${\rm Ind}(G)$ and ${\rm Ind}(H)$ are both pure.

If a simplicial complex is pure, it may indicate that the complex
has a richer combinatorial or topological structure.  Two examples
relevant to this paper are vertex decomposability or
shellability.   Inspired by Topp and Volkmann's result,
we can
ask if ${\rm Ind}(G)$ and ${\rm Ind}(H)$ are both shellable,
respectively, vertex decomposable, does ${\rm Ind}(G[H])$ also
inherit this property?   The purpose of this short note is to prove
that this natural guess is too naive.  In fact, ${\rm Ind}(G[H])$ is
rarely shellable or vertex decomposable.  Precisely, we prove:

\begin{theorem} \label{maintheorem}
Let $G$ and $H$ be finite simple graphs.  
\begin{enumerate}
\item[$(a)$] Suppose that $G$ is a graph of isolated vertices. 
Then ${\rm Ind}(H)$ is vertex decomposable,
respectively shellable, if and only if ${\rm Ind}(G[H])$ is
vertex decomposable, respectively, shellable.
\item[$(b)$] Suppose that $G$ is not a graph of isolated vertices. 
Then ${\rm Ind}(G[H])$ is vertex decomposable if and only if 
${\rm Ind}(G)$ is vertex decomposable and $H = K_m$ for some $m\geq 1$.
\item[$(c)$]  Suppose that $G$ is not a graph of isolated vertices. 
If ${\rm Ind}(G)$ is shellable and $H = K_m$, then
${\rm Ind}(G[H])$ is shellable.  Furthermore, if ${\rm Ind}(G[H])$
is shellable, then $H = K_m$ for some  $m \geq 1$.
\end{enumerate}
\end{theorem}

\noindent
In the above statement, $K_m$ denotes the complete graph on $m$
vertices.
Note that when $G$ is a graph of $a$ isolated vertices,
then $G[H]$ is simply $a$ disjoint copies of $H$. So 
Theorem \ref{maintheorem} $(a)$ will follow from \cite[Theorem 20]{W}.  
Statements
$(b)$ and $(c)$, which are proved in 
Section 3, depend upon results of Hoshino \cite{HPhD}
and  Moradi and Khosh-Ahang \cite{MK}.

We conclude this paper with some applications to circulant graphs.
In particular, starting from a circulant graph
found in \cite{EVMVT}, we construct construct an infinite family of 
graphs with the property that every graph in this family has the 
property that the independence complex of each graph is shellable, but not 
vertex decomposable.  
To the best of our
knowledge, this is the first known infinite family with this property.

%%%%%%%%%%%%%%%%%%%%%%%%%%%%%%%%%%%%%%%%%%%%%%%%%%%%%%%%%%%%%%%%%%%%%%%%%%%%%%%%

\section{Background Definitions and Results}\label{sec:background} 

A simplicial complex $\Delta$ on a vertex set $V = \{x_1,\ldots,x_n\}$
is a subset of $2^V$ such that $(i)$ if $G \subseteq F \in \Delta$,
then $G \in \Delta$, and $(ii)$ $\{x_i\} \in \Delta$ for all $x_i \in V$.
Elements of $\Delta$ are called {\it faces}.  
The maximal faces of $\Delta$ with respect to inclusion are called
the {\it facets} of $\Delta$. 
A simplicial complex is called
{\it pure} if all its facets have the same dimension.  If $F_1,\ldots,F_s$
is a complete list of the facets of $\Delta$, then we sometimes write
$\Delta = \langle F_1,\ldots,F_s \rangle$.  

Given any simplicial complex $\Delta$ and face $F \in \Delta$,
we can create two new simplicial complexes.  The {\it deletion} of $F$
from $\Delta$
is 
${\rm del}_{\Delta}(F) = \{G \in \Delta ~|~ F \not\subseteq G \}.$
The {\it link} of $F$ in $\Delta$ is 
${\rm link}_\Delta(F) = \{G \in \Delta ~|~ F \cap G = \emptyset ~\mbox{and}~ F 
\cup G \in \Delta\}.$
When $F = \{x\}$ for a vertex $x \in V$, we shall abuse notation
and simply write ${\rm del}_{\Delta}(x)$ and ${\rm link}_{\Delta}(x)$.

\begin{definition}  Let $\Delta$ be a pure simplicial complex.
\begin{enumerate}
\item[$(i)$] $\Delta$ is {\it shellable} if there is an
ordering of the facets $F_1,\ldots,F_s$ of $\Delta$ such that for 
all $1 \leq j < i \leq s$, there is some 
$x \in F_i \setminus F_j$ and some $k \in \{1,\ldots,i-1\}$ such 
that $\{x\} = F_i \setminus F_k$.
\item[$(ii)$] $\Delta$ is {\it vertex decomposable} if
$(a)$  $\Delta$ is a simplex (i.e., has
a unique facet), or $(b)$ there exists a vertex $x \in V$
such that ${\rm del}_{\Delta}(x)$ and ${\rm link}_{\Delta}(x)$ are vertex
decomposable.  
\end{enumerate}
\end{definition}

Vertex decomposability and shellability are related as follows:

\begin{lemma}[{\cite[Corollary 2.9]{PB}}]\label{prop}
Let $\Delta$ be a pure simplicial complex.
If $\Delta$ is vertex decomposable, then $\Delta$ is shellable.
\end{lemma}

The independence complex ${\rm Ind}(G)$ of graph $G$ is 
an example of a simplicial complex.  We will
say $G$ is vertex decomposable, respectively shellable,
if ${\rm Ind}(G)$ has this property.  

The following result, due to
Hoshino \cite{HPhD}, is the first of two key critical results
needed to prove Theorem \ref{maintheorem}.
In the proof, 
$\pi_1:V_G \times V_{H}
\rightarrow V_G$ denotes the projection $\pi_1((x_i,y_j)) = x_i$.
In addition, $\alpha(G)$ denotes the cardinality of the
largest independent set of $G$.  

\begin{theorem} \label{nonresult}
Suppose that $G$ is not the graph of isolated vertices.  
If $H \neq K_m$ for some $m\geq 1$, then $G[H]$ is not shellable.
\end{theorem}

\begin{proof}  (Based upon \cite[Theorem 4.52]{HPhD}.)
We note that ${\rm Ind}(G)$ has at least two facets.  Indeed, $G$ must
have at least two vertices that are adjacent, say $x_1$ and $x_2$.
Because $\{x_1\}$ and $\{x_2\}$ are independent sets,
there exists facets of ${\rm Ind}(G)$,
that contain $x_1$ and $x_2$.  Furthermore, these facets must be distinct
because $x_1$ is adjacent to $x_2$.

If $H \neq K_n$, then there are at least two vertices 
in $V_H$ that are not adjacent,
and thus $\alpha(H) \geq 2$.  Furthermore, the construction
of $G[H]$ implies that $\alpha(G[H]) = \alpha(G)\alpha(H)$.

Suppose that ${\rm Ind}(G[H])$ has a shelling.  Let
$F_1,\ldots,F_s$ be the corresponding shelling. 
Because ${\rm Ind}(G[H])$ is shellable, it is pure, which implies
that both ${\rm Ind}(G)$ and ${\rm Ind}(H)$ are pure (this is a restatement
of Topp and Volkmann's \cite{TV} result about well-covered graphs).  So, every
facet of ${\rm Ind}(G[H])$ has cardinality $\alpha(G)\alpha(H)$. 
For each $i\in \{1,\ldots,s\}$, it follows that $\pi_1(F_i)$ is a
maximal independent set of $G$, that is, $\pi_1(F_i) \in {\rm Ind}(G)$.
Because ${\rm Ind}(G)$ has at least two facets, there is an index $k$
such that $\pi_1(F_1) = \cdots = \pi_1(F_{k-1}) \neq \pi_1(F_k)$.
Then, for each $i=1,\ldots,k-1$, we have
\[|F_i \cap F_k| \leq (\alpha(G)-1)\alpha(H) < \alpha(G)\alpha(H)-1 = |F_k|-1\]
where the strict inequality follows from the fact that $\alpha(H) \geq 2$.

However, because $F_1,\ldots,F_s$ is a shelling order, for every 
$1 \leq j <  k$, there exists some $x \in F_k \setminus F_j$ such that
$\{x\} = F_k \setminus F_i$ for some $1 \leq i <k$.  Because
$F_k$ and $F_i$ have the same cardinality, this implies that
$|F_i \cap F_k| = |F_k|-1$, which contradicts the inequality given
above.   So, ${\rm Ind}(G[H])$ cannot be shellable if $H \neq K_m$.
 \end{proof}

Moradi and Khosh-Ahang \cite{MK} introduced the 
expansion of the simplicial complex.  Although their results apply
to any simplicial complex, we only present their results for
independence complexes.  We first define the expansion of a graph.

\begin{definition}
Let $G$ be a graph on the vertex set $V = \{x_1,\ldots,x_n\}$
and let $(s_1,\ldots,s_n) \in \mathbb{N}^n_{>0}$ be an $n$-tuple
of positive integers.  The {\it $(s_1,\ldots,s_n)$-expansion of $G$},
denoted $G^{(s_1,\ldots,s_n)}$, is the graph on the vertex set
$V_{G^{(s_1,\ldots,s_n)}} = \{x_{1,1},\ldots,x_{1,s_1},x_{2,1},\ldots,
x_{2,s_2},\ldots,x_{n,1},\ldots,x_{n,s_n}\}$
with edge set 
$E_{G^{(s_1,\ldots,s_n)}} = \{\{x_{i,j},x_{k,l}\} ~|~ \{x_i,x_k\} \in E_G ~~\mbox{or}
~~i = k\}.$
\end{definition}

The next two results now follow from more general results
of Moradi and Khosh-Ahang.

\begin{theorem} \label{MoradiKhoshAhang}
Let $G$ be a finite simple graph and $(s_1,\ldots,s_n)
\in \mathbb{N}^{n}_{>0}$.
\begin{enumerate}
\item[$(i)$]\cite[Theorem 2.7]{MK} $G$ is vertex decomposable if and only if 
$G^{(s_1,\ldots,s_n)}$ is vertex decomposable.
\item[$(ii)$]\cite[Theorem 2.12]{MK} If $G$ is shellable, then $G^{(s_1,\ldots,s_n)}$ is shellable.
\end{enumerate}
\end{theorem}

%%%%%%%%%%%%%%%%%%%%%%%%%%%%%%%%%%%%%%%%%%%%%%%%%%%%%%%%%%%%%%%%%%%%%%%

\section{Proof of Theorem 1.1}

\begin{proof}(of Theorem \ref{maintheorem})
Statement $(a)$ was already proved in the introduction.  To prove
$(b)$ and $(c)$, observe that
if $G$ has $n$ vertices and $(m,\ldots,m) \in \mathbb{N}^n_{>0}$,
then $G^{(m,\ldots,m)} = G[K_m]$.

$(b)$  Suppose that $G[H]$ is vertex decomposable.  Because
$G[H]$ is also shellable by Theorem \ref{prop}, Theorem \ref{nonresult}   
implies that $H = K_m$ for some $m \geq 1$.  So $G[H] = G^{(m,\ldots,m)}$.
Because $G^{(m,\ldots,m)}$ is vertex decomposable, $G$ is vertex
decomposable by
Theorem \ref{MoradiKhoshAhang}.  For
the converse, because $H = K_m$ and $G$ is vertex decomposable,
$G[H] = G^{(m,\ldots,m)}$ is vertex decomposable by Theorem \ref{MoradiKhoshAhang}.

$(c)$  Because $H = K_m$ and $G$ is shellable,
$G[H] = G^{(m,\ldots,m)}$ is shellable by  Theorem \ref{MoradiKhoshAhang}.  
As well, if $G[H]$ is shellable, we must have $H = K_m$ for some 
$m \geq 1$ by Theorem \ref{nonresult}.
\end{proof}

%%%%%%%%%%%%%%%%%%%%%%%%%%%%%%%%%%%%%%%%%%%%%%%%%%%%%%%%%%%%%%%%%%%%%%

\section{Applications to circulant graphs}

We define a {\it circulant graph} on $n \geq 1$ vertices as follows.
Let $S \subseteq \{1,2,\ldots,\floor{\frac{n}{2}}\}$. The circulant graph 
$C_n(S)$ is the graph with $V = \{x_0, x_1,\ldots,x_{n-1}\}$, such that 
$\{x_a, x_b\}$ is an edge of $C_n(S)$ if 
and only if $|a-b| \in S$ or $n-|a-b| \in S$.
See
\cite{BGM,Brown11,EVMVT,HPhD,VVW} for some recent
papers on the properties of $C_n(S)$, especially well-covered circulant 
graphs.

Hoshino proved the following result about the lexicographical
products of circulant graphs (in fact, the original result
describes how to construct the lexicographical product from the
data describing the two initial circulant graphs).

\begin{theorem}\label{buildcirculants} \cite[Theorem 2.31]{HPhD}
Let $G = C_{n}(S_1)$ and $H = C_{m}(S_2)$ be circulant graphs.
Then $G[H]$ is also a  circulant graph.
\end{theorem}

Because $K_m$ is the circulant graph $K_m = 
C_n(\{1,2,\ldots,\lfloor \frac{m}{2} \rfloor\})$, Theorem
\ref{maintheorem}, combined with
the Theorem \ref{buildcirculants} implies the following
result.

\begin{theorem}\label{corcirculant}
Let $G$ be a circulant graph such that $G$ is vertex
decomposable, respectively shellable.  Then $G[K_m]$ with
$m\geq 1$ is a circulant graph that is also vertex decomposable,
respectively, shellable.
\end{theorem}

\begin{remark}
Many families of vertex decomposable and shellable circulant graphs
have been  identified \cite{EVMVT,VVW}.  From any such
graph, we can now build an infinite family of circulant graphs
that is either vertex decomposable or shellable using the above
result.
\end{remark}

It has long been known that the converse of Theorem \ref{prop} is false
(see \cite{PB}).   However,
it was less clear whether the converse of Theorem \ref{prop} was still
false if we restricted to independence complexes of graphs. 
To the best of our knowledge, the circulant graph
$C_{16}(1,4,8)$ found in \cite[Theorem 6.1]{EVMVT}
is the first example of a graph that is shellable
but not vertex decomposable.
By combining Theorem \ref{MoradiKhoshAhang} with this example,
we have an infinite family of 
independence complexes which are shellable but not vertex decomposable.
In addition, Theorems \ref{maintheorem} and 
\ref{corcirculant} allow us to make an infinite
family of circulant graphs with this property.  

\begin{theorem}\label{nonexamples}
Let $G = C_{16}(1,4,8)$ and $(s_1,\ldots,s_n) \in \mathbb{N}^{n}_{>0}$.
Then $G^{(s_1,\ldots,s_n)}$ is shellable but not vertex decomposable.
Furthermore, for all $m \geq 1$, $G^{(m,\ldots,m)} = G[K_m]$ is a circulant
graph that is shellable but not vertex decomposable.
\end{theorem}

%\begin{proof}
%By \cite[Theorem 6.1]{EVMVT}, the graph $G$ is shellable,
%so $G^{(s_1,\ldots,s_n)}$ is shellable by Theorem \ref{MoradiKhoshAhang}.
%Because $G$ is not vertex decomposable,
%Theorem \ref{MoradiKhoshAhang} again implies that $G^{(s_1,\ldots,
%s_n)}$ cannot be vertex decomposable.   The second statement
%follows from Theorem \ref{corcirculant}
%\end{proof}

\begin{remark}
We end with a couple of concluding remarks.
The most obvious question to ask is if the
converse of Theorem \ref{maintheorem} $(c)$ holds, or more
generally, does the converse of  \cite[Theorem 2.12]{MK} hold.
To prove the converse, we would need to determine if ${\rm Ind}(G[K_m])$
being shellable implies that $G$ 
is shellable.  

Our strategy to construct an infinite family of shellable but not 
vertex decomposable graphs is to find an initial graph with this property,
and then apply Theorem \ref{MoradiKhoshAhang}.  However, finding
the initial graph with this property is quite difficult.  Besides
the graph $G = C_{16}(1,4,8)$, we know of only one other graph
with this property, namely the circulant graph $C_{20}(1,5,10)$,
which was verified computationally using {\it Macaulay2} \cite{Mt}.  
We were also able
to computationally verify that $C_{24}(1,6,12)$ is not vertex decomposable,
although we have not verified it is shellable (it is Cohen-Macaulay).
Based upon on this very slim evidence, we suspect that 
the graphs $G = C_{4s}(1,s,2s)$ with $s \geq 4$ are
shellable but not vertex decomposable.   The first three graphs
in this family can be seen in Figure \ref{family}.
\begin{figure}
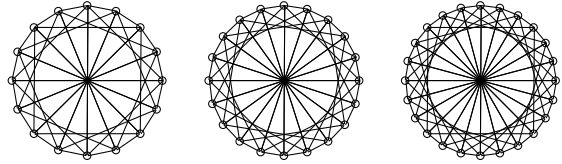

\[\Circulant{16}{1,4,8}
\hspace{.5cm}
\Circulant{20}{1,5,10}
\hspace{.5cm}
\Circulant{24}{1,6,12}\]
\caption{The circulant graphs $C_{16}(1,4,8), C_{20}(1,5,10),$ and $C_{24}(1,6,12)$}
\label{family}
\end{figure}
\end{remark}

%%%%%%%%%%%%%%%%%%%%%%%%%%%%%%%%%%%%%%%%%%%%%%%%%%%%%%%%%%%%%%%%%%%%%%%%%%%%%%%%%%%%%

%%%%%%%%%%%%%%%%%%%%%%%%%%%%%%%%%%%%%%%%%%%%%%%%%%%%%%%%%%%%%%

\end{document}